\documentclass{amsart}
\usepackage{amssymb,latexsym}
\usepackage{amsmath}
\usepackage{graphicx}
\usepackage{color}
\usepackage{pinlabel}

\newtheorem{thm}{Theorem}[section]

\newtheorem{cor}[thm]{Corollary}
\newtheorem{defin}[thm]{Definition}
\newtheorem{lemma}[thm]{Lemma}
\newtheorem{prop}[thm]{Proposition}

\newcommand{\bdd}{\mbox{$\partial$}}

\newcommand{\aaa}{\mbox{$\alpha$}}

\newcommand{\czero}{\mbox{$c \times \{ 0 \}$}}
\newcommand{\cone}{\mbox{$c \times \{ 1 \}$}}
\newcommand{\Fzero}{\mbox{$F \times \{ 0 \}$}}
\newcommand{\Fone}{\mbox{$F \times \{ 1 \}$}}
\newcommand{\Fz}{\mbox{$(F \times I)_{surg}$}}
\newcommand{\Fempty}{\mbox{$(F \times I)_{\emptyset}$}}
\newcommand{\FSempty}{\mbox{$(F \times S^1)_{\emptyset}$}}
\newcommand{\FSz}{\mbox{$(F \times S^1)_{surg}$}}
\newcommand{\Fp}{\mbox{$(F \times I)^{+}$}}
\newcommand{\Fpz}{\mbox{$(F \times I)^{+}_{surg}$}}
\newcommand{\Fpempty}{\mbox{$(F \times I)^{+}_{\emptyset}$}}

\begin{document}



\thanks{Research partially supported by NSF grants.}

\title{Surgery on a knot in  Surface $\times$ I }

\author{Martin Scharlemann}
\address{\hskip-\parindent
Martin Scharlemann\\
Mathematics Department \\
University of California, Santa Barbara \\
Santa Barbara, CA 93117, USA}
\email{mgscharl@math.ucsb.edu}

\author{Abigail Thompson}
\address{\hskip-\parindent
Abigail Thompson\\
Mathematics Department \\
University of California, Davis\\
Davis, CA 95616, USA}
\email{thompson@math.ucdavis.edu}

\date{\today}

\begin{abstract}  Suppose $F$ is a compact orientable surface, $K$ is a knot in $F \times I$, and $(F \times I)_{surg}$ is the $3$-manifold obtained by some non-trivial surgery on $K$.   If $F \times \{ 0 \}$ compresses in $(F \times I)_{surg}$, then there is an annulus in $F \times I$ with one end $K$ and the other end an essential simple closed curve in $F \times \{ 0 \}$.  Moreover, the end of the annulus at $K$ determines the surgery slope.  

An application:  suppose $M$ is a  compact orientable $3$-manifold that fibers over the circle.  If surgery on $K \subset M$ yields a reducible manifold, then either 
\begin{itemize}
\item the projection $K \subset M \to S^1$ has non-trivial winding number,
\item $K$ lies in a ball,
\item $K$ lies in a fiber, or
\item $K$ is cabled
\end{itemize}
 \end{abstract}

\maketitle

The study of Dehn surgery on knots in $3$-manifolds has a long and rich history, interacting in a deep way with 
\begin{itemize}
\item sophisticated combinatorics (\cite{GL},  \cite{CGLS}),
\item the theory of character varieties (\cite{CGLS}, \cite{BGZ}), and
\item sutured manifold theory (\cite{Ga1}, \cite{Sch})
\end{itemize}

It is pleasing then to find a result that is simple to state, easy to understand and yet  has so far escaped explicit notice.  Yi Ni has pointed out that there is at least implicit overlap of our results with \cite[Theorem 1.4 and Section 3]{Ni}.  

\begin{thm} \label{thm:main} Suppose $F$ is a compact orientable surface, $K$ is a knot in $F \times I$, and $(F \times I)_{surg}$ is the $3$-manifold obtained by some non-trivial surgery on $K$. If $F \times \{ 0 \}$ compresses in $(F \times I)_{surg}$, then $K$ is parallel to an essential simple closed curve in $F \times \{ 0 \}$.   Moreover, the annulus that describes the parallelism determines the slope of the surgery.
\end{thm}

An important precursor to this theorem is \cite[Proposition 4.6]{BGZ}.  This proposition examines the same sort of question, but restricted to the case in which $F$ is closed, and concludes that the slope of the surgery must be integral.  It does not directly comment on the position of the knot in $F \times I$, though the proof itself comes close: the proof of that proposition offers our conclusion (that  $F \times I - \eta(K)$ is what the authors call a ``hollow product'') as one of two possibilities.  Our argument is independent of this result, resting entirely on central sutured manifold results from \cite{Ga2}, \cite{Ga3} and \cite{Sch}.

\medskip

First some notation:

\begin{defin} Let $F$ be an orientable surface.  A simple closed curve in $F$ is {\bf trivial} if it bounds a disk in $F$.  A properly embedded arc  (resp. non-trivial simple closed curve) $\aaa \subset F$ is {\bf essential} if it is not parallel to an arc in $\bdd F$ (resp. component of $\bdd F$).  

An annulus $S^1 \times (I, \bdd I) \subset F \times (I, \bdd I)$ is an {\bf  essential spanning annulus} in $F \times I$ if it is properly isotopic to $\aaa \times I$, for some essential simple closed curve $\aaa$ in $F$. 

A square $(I, \bdd I) \times (I, \bdd I) \subset (F, \bdd F) \times (I, \bdd I)$ is an {\bf  essential spanning square} in $F \times I$ if it is properly isotopic to $\aaa \times I$, for some essential properly embedded arc $\aaa$ in $F$. 

\end{defin}

For example, if a fundamental class $[\aaa ] \in H_1(F, \bdd F)$ of a properly embedded arc $\aaa \subset F$ is non-trivial then $\aaa$ is essential.  In fact, by Poincare duality, $\aaa$ is non-separating in $F$.

\bigskip

\noindent {\small {\bf Acknowledgement:} We would like to thank Cameron Gordon and Ying-Qing Wu for helpful comments on a preliminary draft.}

\section{A review of Gabai: When $\chi(F) = 0$}

Consider the following special case of a theorem of Gabai \cite{Ga3}:

\begin{prop}  \label{prop:annulus} For $A$ an annulus, suppose $K$ is a knot in $A \times I$.  Let $(A \times I)_{surg}$ be the manifold obtained as the result of some non-trivial Dehn surgery on $K$.  If $A \times \{ 0 \}$ compresses in $(A \times I)_{surg}$, then $K$ is parallel in $A \times I$ to the core curve $\aaa$ of $A \times \{ 0 \}$.  Moreover, the annulus that describes the parallelism determines the slope of the Dehn surgery.
\end{prop}

\begin{proof} The generating homology class $[\aaa]$ of $H_1(A \times \{ 0 \})$ is non-trivial in $H_1(A \times I - K)$ but trivial in $H_1((A \times I)_{surg})$, so a simple homology argument shows that the fundamental class $[K]$ is a generator of $H_1(A \times I)$ and that the surgery slope is some longitude of $K$.  Thus $H_1((A \times I)_{surg})$ does not contain a finite summand, so only the first conclusion of \cite[Theorem 1.1]{Ga3} is possible: $K$ is a braid in the solid torus $A \times I$.  But the only braid with winding number $1$ in a solid torus is a core of the solid torus, so in fact $A \times I - \eta(K)$ is just a collar of $\bdd \eta(K)$.   Given this collar structure, the only way that $A \times \{ 0 \}$ could compress after Dehn surgery on $K$ (i. e. after attaching a solid torus to $\bdd \eta(K) \subset A \times I - \eta(K)$) is if the surgery slope is parallel, via the collar $A \times I - \eta(K)$, to the core of $A \times \{ 0 \}$.  
\end{proof}

Another theorem of Gabai gives an analogous theorem for a torus $T$:

\begin{prop}  \label{prop:torus} Suppose $K$ is a knot in $T \times I$.  Let $(T \times I)_{surg}$ be the manifold obtained as the result of non-trivial Dehn surgery on $K$.  If $T \times \{ 0 \}$ compresses in $(T \times I)_{surg}$, then $K$ is parallel to an essential simple closed curve in $T \times \{ 0 \}$.  Moreover, the annulus that describes the parallelism determines the slope of the surgery.
\end{prop}

\begin{proof} The hypothesis guarantees that the fundamental class $[K]$ is non-trivial in $H_1(T)$ so $T \times I - \eta(K)$ is an irreducible, $\bdd$-irreducible $3$-manifold whose boundary consists of tori.  We can then regard $(T \times I)_{\emptyset} = T \times I - \eta(K)$ as a taut sutured manifold. (See \cite{Ga1}, \cite{Ga2}.)  From that point of view, $(T \times I)_{surg}$ is the sutured manifold obtained by a non-trivial filling of $\bdd \eta(K) \subset \bdd (T \times I)_{\emptyset}$, but it is not taut, since the boundary component $T \times \{ 0 \}$ compresses in  $(T \times I)_{surg}$.  

Let $c \subset T \times \{ 0 \}$ be an essential simple closed curve so that $[K]$ is a multiple of the fundamental class $[c]$ in $H_1(T \times I) $. The homology class $[c] \times [I, \bdd I] \in H_2(T \times I, \bdd (T \times I))$ is represented by the spanning annulus $c \times I$ and so has trivial Thurston norm.  Since $[K]$ is a multiple of $[c]$ in $H_1(T \times I) \cong H_1(T)$ it follows that the algebraic intersection $[K] \cdot ([c] \times [I, \bdd I])$ is trivial.  In particular, the homology class $[c] \times [I, \bdd I]$ lifts to a homology class  $\beta \in H_2((T \times I)_{\emptyset})$.  Since the non-trivial filling of $\bdd \eta(K)$ destroys tautness, it follows from \cite[Corollary 2.4]{Ga2} that trivial filling of $\bdd \eta(K)$, to get $T \times I$, does not lower the Thurston norm of $\beta$.  Hence the Thurston norm of $\beta$ is trivial.  In particular, it can be represented by spanning annuli (in fact an essential spanning annulus) in $(T \times I)_{\emptyset}$, i. e. by a spanning annulus in $T \times I$ that is disjoint from $K$.  

Since this essential spanning annulus (which we can take to be $c \times I$) in $T \times I$ is disjoint from $K$, it persists into $(T \times I)_{surg}$. A standard outermost arc argument shows that there is a compressing disk for $T \times \{ 0 \}$ in $(T \times I)_{surg}$ that is also disjoint from $c \times I$.  Thus we can apply Proposition \ref{prop:annulus} to the annulus $A = T - \eta(c)$ and reach the required conclusion.  
\end{proof}

Our goal is to prove the identical result for surgery on a knot $K \subset F \times I$ when $F$ is any compact orientable surface.  The philosophy of the proof is captured above:  Use sutured manifold theory to find an essential spanning annulus or square in $F \times I$ that is disjoint from the knot.  Cut open $F$ along this essential annulus or square to give a surface $F'$ that is simpler.  Continue until the surface becomes an annulus, and apply Proposition \ref{prop:annulus} .

A difficulty in the above approach is that for $F$ more complicated than a torus, particularly when $F$ has boundary, then \cite[Corollary 2.4]{Ga2} does not apply.  Its proof requires in an important way that whatever surface we are using to determine Thurston norm (in our context the spanning annulus or square) has its boundary lying only on tori boundary components of the ambient $3$-manifold (in our context, $F \times I$).  Overcoming this difficulty requires some trickery and the use of the central theorem of \cite{Sch}.

\section{Foundational lemmas, useful when $\chi(F) < 0$}

For $F$ a compact orientable surface and $K \subset F \times I$, let $\Fempty$ denote $F \times I - \eta(K)$ and let  $\Fz$ denote the manifold obtained from $F \times I$ by non-trivial surgery on $K$. See the schematic in Figure \ref{fig:FemptyFz}.

 \begin{figure}[ht!]
 \labellist
\small\hair 2pt
\pinlabel $F$ at 81 103
\pinlabel $I$ at 163 55
\pinlabel ${F \times I}$ at 81 9
\pinlabel ${(F \times I)_{\emptyset}}$ at 280 9
\pinlabel $\eta(K)$ at 280 108
\pinlabel ${\Fz}$ at 447 9
 
 \endlabellist
    \centering
    \includegraphics[scale=0.7]{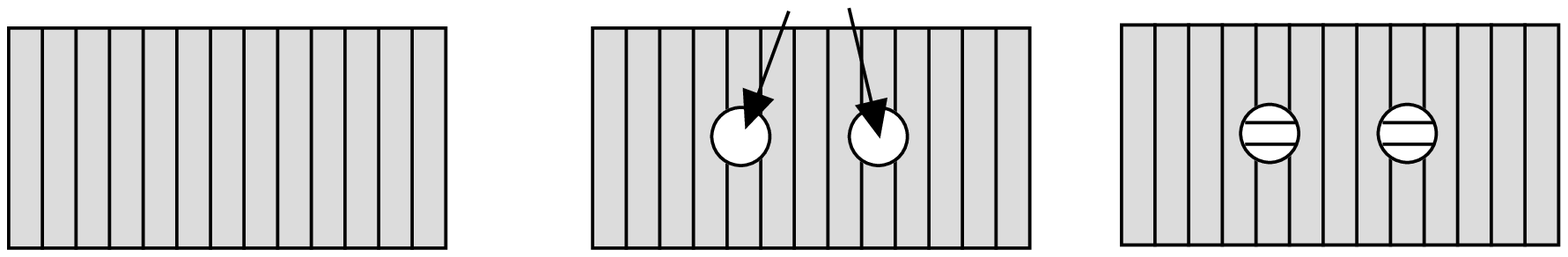}
    \caption{} \label{fig:FemptyFz}
    \end{figure}

\begin{lemma} \label{lemma:found1} Suppose $F$ is not an annulus and there is a non-trivial simple closed closed curve $c \subset F$ so that both $\czero \in \Fzero$ and $\cone \in \Fone$ bound disks in $\Fz$.  Then there is an essential spanning annulus or essential spanning square in $F \times I$ that is disjoint from $K$.
\end{lemma}

\begin{proof}    If  $\Fempty$ were reducible, then $K$ would lie inside a $3$-ball in $F \times I$, and surgery on $K$ could not make $F \times \bdd I$ compressible.  So $\Fempty$ is irreducible.  If $K$ is a satellite knot (that is, $K$ lies in a solid torus $L \subset F \times I$ so that $\bdd L$ is essential in $\Fempty$) the argument below could be applied to $L$ instead of $\eta(K)$ with the same result.  So henceforth we also assume that $K$ is not a satellite knot in $F \times I$.

 Create $F \times S^1$ by identifying $\Fzero$ with $\Fone$ and let $\FSz$ (resp. $\FSempty$) denote $\Fz$ (resp. $\Fempty$) with its boundary components $\Fzero$ and $\Fone$ identified.  Stated other ways, $\FSempty = (F \times S^1) - \eta(K)$ and $\FSz$ is the manifold obtained from $F \times S^1$ by Dehn surgery on $K \subset F \times S^1$ or the manifold obtained by non-standard filling of $\bdd \eta(K) \subset \bdd \FSempty$.   Any disk or incompressible annulus in $F \times I$ that has its entire boundary on $F \times \{ 0 \}$ or on $F \times \{ 1 \}$ is boundary parallel in $F \times I$.  It follows immediately that $\FSempty$ is also irreducible and, from the similar assumption on $\Fempty$, that $K$ is not a satellite knot in $F \times S^1$.


Let $D_0, D_1 \subset \Fz$ be disks bounded by $\czero$ and $\cone$ respectively.  Consider the sphere $S \subset \FSz$ obtained by identifying the boundaries of $D_0$ and $D_1$.  See the schematic in Figure \ref{fig:FSempty}.

 \begin{figure}[ht!]
 \labellist
\small\hair 2pt
\pinlabel ${c \times \{ 0 \}}$ at 288 10
\pinlabel ${c \times \{ 1 \}}$ at 288 145
\pinlabel $\FSempty$ at -41 78
\pinlabel $\FSz$ at 410 78
 
 \endlabellist
    \centering
    \includegraphics[scale=0.7]{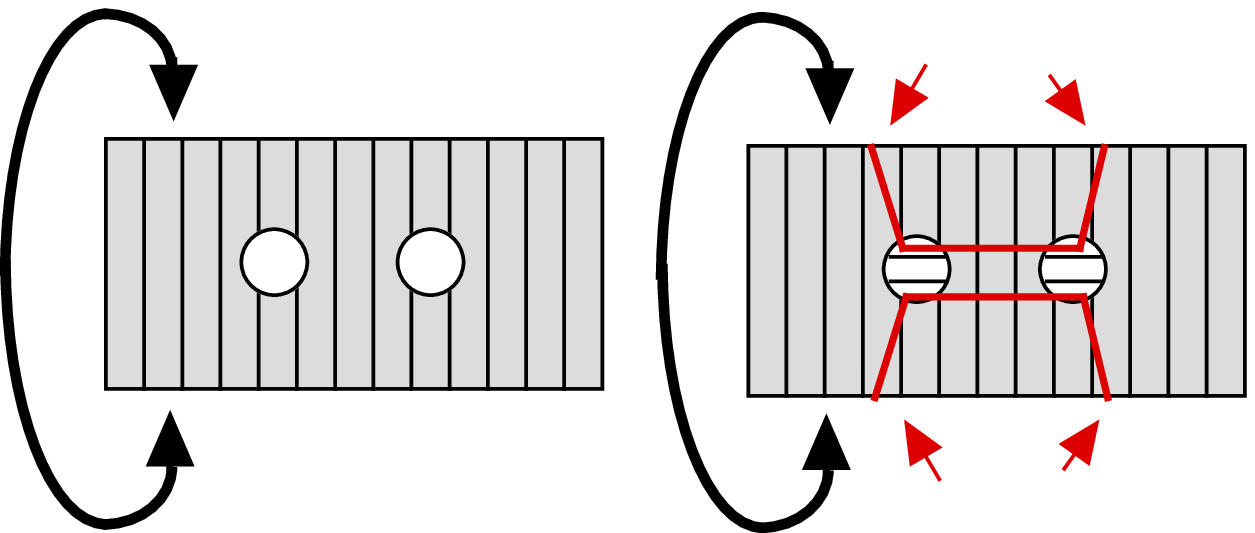}
    \caption{} \label{fig:FSempty}
    \end{figure}
    
    \bigskip

\noindent{\bf Claim:} $S$ is a reducing sphere

\bigskip

The inclusion homomorphism $H_1(F) \to H_1(\FSempty)$ is injective, since $H_1(F) \to H_1(F \times S^1)$ is.  It follows that $nullity(H_1(F) \to H_1(\FSz)) \leq 1$.  But if $S$ bounded a ball in $\FSz$ that ball would properly contain a component of $F - c$ that is bounded by $c$.  Since $c$ is non-trivial in $F$, that component would not be a disk, so it would have positive genus.  In particular, at least a rank two summand of $H_1(F)$ would be trivial in $H_1(\FSz)$.    The contradiction proves the claim.

\bigskip

Orient $K$ and $S^1$ and let $[K]$ (resp $[S^1]$) denote the homology class represented by $K$ in $H_1(F)$ (resp. $\{ point \} \times S^1$ in $H_1(F \times S^1$)). Since $F$ is not an annulus, $rank(H_1(F, \bdd  F)) \geq 2$.  By Poincare duality, there is an  $\alpha \neq 0 \in H_1(F, \bdd F)$ so that $\alpha \cdot [K] = 0$.  $\aaa$ can be represented by essential circles and arcs in $F$, so $\aaa \times [S^1] \in H_2((F, \bdd F) \times S^1)$ can be represented by essential annuli and tori.   In particular, $\aaa \times [S^1]$ has trivial Thurston norm.  Since $\alpha \cdot [K] = 0$ it follows that  $\aaa \times [S^1]$ is the image of some $\beta \in H_2(\FSempty, \bdd F \times S^1)$.  

The manifold $F \times S^1$ is an irreducible manifold with only tori boundary components, so it can be viewed as a taut sutured manifold.  Since non-trivial surgery on $K \subset F \times S^1$ gives a reducible manifold, it follows, essentially from \cite[Corollary 2.4]{Ga2} that trivially replacing $\eta(K)$ in $\FSempty$, which yields $F \times S^1$, does not reduce Thurston norm of any homology class.  In particular, $\beta$ can be represented by tori and annuli in $\FSempty$.  These tori and annuli can be isotoped to intersect the incompressible $F \subset (F \times S^1) - \eta(K)$ in circles that are parallel in the tori and and arcs that are spanning in the annuli, so the intersection of these surfaces with $\Fempty$ are properly embedded annuli and squares. Since $\aaa \times [S^1] \in H_2(F \times I)$ is non-trivial, it is easy to see that at least one of the annuli and squares in $\Fempty$ must be essential and spanning in $F \times I \supset \Fempty$.
\end{proof}

\begin{lemma}  \label{lemma:found2a} Suppose $c$ is a non-trivial curve in $F$ so that $\czero \subset \Fzero$ bounds a disk in $\Fz$. Let $\Fpempty$ be the manifold obtained from $\Fempty = F \times I - \eta(K)$ by attaching a $2$-handle to $\czero$.  Then any curve in $\Fone$ that compresses in $\Fpempty$ also compresses in $\Fz$.
\end{lemma} 

\begin{proof}  Let $c' \subset F$ be a non-trivial curve so that $c' \times \{ 1 \}$ compresses in $\Fpempty$.  The intersection with $\Fempty$ of the compressing disk  is a planar surface with one boundary component parallel $c' \times \{ 1 \}$ in $\Fone$ and the other boundary components all parallel to $\czero$ in $\Fzero$.  But the latter are all null-homotopic in $\Fz$, so it follows that $c' \times \{ 1 \}$ is nullhomotopic in $\Fz$, as required.  See Figure \ref{fig:Fpempty}.
\end{proof}

 \begin{figure}[ht!]
 \labellist
\small\hair 2pt
\pinlabel ${c \times \{ 0 \}}$ at 259 20
\pinlabel ${c' \times \{ 1 \}}$ at 36 151
\pinlabel $\Fpempty$ at -33 86
\pinlabel $\Fpz$ at 375 86
 
 \endlabellist
    \centering
    \includegraphics[scale=0.7]{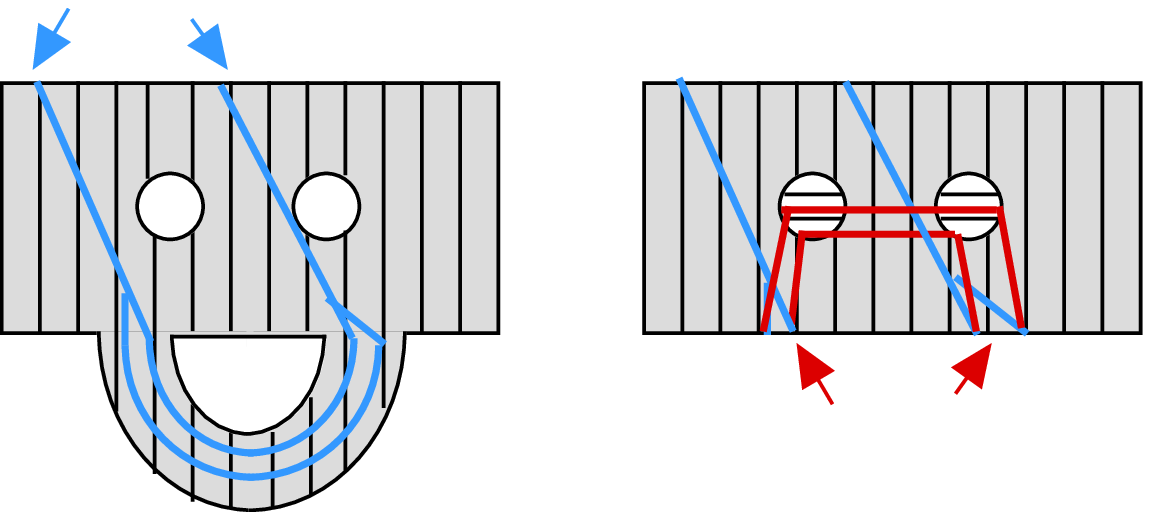}
    \caption{} \label{fig:Fpempty}
    \end{figure}

\begin{lemma}  \label{lemma:found2b} Suppose $c$ is a non-trivial separating curve in $F$ so that $\czero \subset \Fzero$ bounds a disk in $\Fz$. Let $\Fp$ (resp. $\Fpempty$) be the manifold obtained from $F \times I $ (resp $\Fempty$) by attaching a $2$-handle to $\czero$.  Then $\Fone$ is compressible in $\Fpempty$.
\end{lemma} 

\begin{proof}   If $\Fpempty$ is reducible, then $K$ lies in a $3$-ball, so it follows immediately that $\cone$ compresses in $\Fpempty$.  So henceforth we may as well assume $\Fpempty$ is irreducible.

Suppose, towards a contradiction, that $\Fone$ is incompressible in $\Fpempty$.  $\Fp$ can be dually viewed as the manifold obtained from $[\bdd \Fp - \Fone] \times I$ by attaching a single $1$-handle. It follows that $\bdd \Fp - \Fone$ is incompressible in  $\Fp$ and so also in $\Fpempty$.  Combining, we have that $\bdd \Fp - (\bdd F \times I)$ is incompressible in $\Fpempty$.

\bigskip

\noindent{\bf Case 1:} F is closed, so in fact $\bdd \Fp$ is incompressible in $\Fpempty$.

\bigskip

Apply \cite{Sch}, with $M$ the manifold $\Fpz$ and $M' = \Fp$.  By the hypothesis of this case, $M - K = \Fpempty$ is irreducible and $\bdd$-irreducible.  On the other hand, since $\czero$ bounds a disk in $\Fz$,  $M = \Fpz$ contains a $2$-sphere that passes exactly once through the $2$-handle, so the $2$-sphere does not bound a rational homology ball.  Since $M' = \Fp$ is $\bdd$-reducible, the only possible conclusion from \cite{Sch} is that $K$ is cabled in $M$, with surgery slope the slope of the cabling annulus.  But the effect of such a surgery would be to create a Lens space summand in $M' = \Fp$ and this is impossible for simple homology reasons we now describe.

Since $c$ is separating $c$ describes a connected sum decomposition $F = F _1 \#_c F_2$.  Moreover, the manifold $\Fp$ deformation  retracts to the $1$-point union $F_1 \vee F_2$.  Hence $H_1(\Fp) \cong H_1(F_1 \vee F_2)$ is free.

\bigskip

\noindent{\bf Case 2:} F has boundary and $\bdd \Fp - (\bdd F \times I)$ is incompressible in $\Fpempty$.

\bigskip

In this case let $M$ be the manifold obtained by attaching a copy of $\Fempty$ to $\Fpz$ along $\bdd F \times I$ and $M'$ be the manifold obtained by attaching a copy of $\Fempty$ to $\Fp$ along $\bdd F \times I.$  Observe that $M-K$ is the union of $\Fempty$ with $\Fpempty$ along $\bdd F \times I$.  Both $\Fempty$ and $\Fpempty$ are irreducible and the complement of $\bdd F \times I$ in the boundary of both $\Fempty$ with $\Fpempty$ is incompressible.  It follows from an innermost disk, outermost arc argument that $M - K$ is irreducible and $\bdd$-irreducible.  Now apply \cite{Sch}, obtaining essentially the same contradiction as in the previous case.
\end{proof}


\begin{lemma}  \label{lemma:found3} Suppose both $\Fzero$ and $\Fone$ are compressible in $\Fz$.  Then there is a non-trivial simple closed curve $c \subset F$ so that both $\czero$ and $\cone$ bound disks in $\Fz$.
\end{lemma}

\begin{proof}  This is obvious if $F$ is an annulus and a simple homology argument establishes the result when $F$ is a torus, so we take $\chi(F) < 0$.  By hypothesis there is a simple closed curve $c$ in $F$ so that $\czero \subset \Fzero$ bounds a disk in $\Fz$ but not in $\Fzero$.  Since $\chi(F) < 0$, we may as well take $c$ to be separating.  As in the proof of Lemma \ref{lemma:found2b}, let $\Fp$ (resp. $\Fpempty$) be the manifold obtained from $F \times I $ (resp $\Fempty = F \times I - \eta(K)$) by attaching a $2$-handle to $\czero$.  Following Lemma \ref{lemma:found2b}, there is a non-trivial simple closed curve $c' \subset F$ so that $c' \times \{1 \}$ bounds a disk in $\Fpempty$.  Isotope $c'$ in $F$ so that it intersects $c$ minimally.  

We first show that $c'$ is parallel to $c$ in $F$.  As noted in the proof of Lemma \ref{lemma:found2a}, since $c$ is separating it describes a connected sum decomposition $F = F _1 \#_c F_2$ and $\Fp$ deformation  retracts to the $1$-point union $F_1 \vee F_2$.      If $c'$ were disjoint from $c$ but not parallel to $c$ then $c'$ would represent a non-trivial element in one of the $\pi_1(F_i)$ and so could not be null-homotopic in $\pi_1(\Fp)$.  This contradicts the definition of $c'$ as a curve null-homotopic in $\Fpempty \subset \Fp$.  Similarly, if $c'$ intersects $c$, then each arc of $c' - c$, being essential in one of the punctured surfaces $F - c$, also represents a non-trivial element in either $\pi_1(F_1)$ or $\pi_1(F_1)$.  This describes $c'$ as a non-trivial word in the free product $\pi_1(F_1) * \pi_1(F_2) \cong \pi_1(F_1 \vee F_2) \cong \pi_1(\Fp)$, with the same contradiction.  

The only remaining possibility is that $c' \times \{1 \}$ is isotopic to $\cone$ in $\Fone$ so $\cone$ also bounds a disk in $\Fpempty$. The result then follows from Lemma \ref{lemma:found2a}.  \end{proof}

\section{The main theorem}

\noindent {\bf Theorem \ref{thm:main}} {\em  Suppose $F$ is a compact orientable surface, $K$ is a knot in $F \times I$, and $(F \times I)_{surg}$ is the $3$-manifold obtained by some non-trivial surgery on $K$.  If $F \times \{ 0 \}$ compresses in $(F \times I)_{surg}$, then $K$ is parallel to an essential simple closed curve in $F \times \{ 0 \}$.   Moreover, the annulus that describes the parallelism determines the slope of the surgery.}

\medskip

\begin{proof}  We may as well assume $F$ is connected and the hypothesis guarantees that $K$ is not contained in a $3$-ball, so $F \times I - \eta(K)$ is irreducible. The proof is by induction on $rank(H_1(F))$. The case $rank(H_1(F)) = 1$ is covered by Propositions \ref{prop:annulus} and \ref{prop:torus}, so we henceforth assume that $rank(H_1(F)) \geq 2$ i. e. $\chi(F) < 0$.  

Since $F \times \{ 0 \}$ compresses in $(F \times I)_{surg}$ and $\chi(F) < 0$, there is a non-trivial separating simple closed curve in $F \times \{ 0 \}$ that bounds a disk in $(F \times I)_{surg}$.  Following Lemmas \ref{lemma:found2a} and \ref{lemma:found2b}, $\Fone$ is also compressible in $(F \times I)_{surg}$.  Following Lemmas  \ref{lemma:found3} and \ref{lemma:found1} there is an essential properly embedded arc or simple closed curve $\aaa \subset F$ so that $K$ is disjoint from $\aaa \times I \subset F \times I$.  A standard innermost disk, outermost arc argument shows that there is a compressing disk for $F \times \{ 0 \}$ in $(F \times I)_{surg}$ that is also disjoint from $\aaa \times I$.  It follows that the hypothesis is still satisfied for $F' = F - \eta(\aaa)$. By inductive assumption (if $\aaa$ is separating, consider just the component of $F' \times I$ that contains $K$), the theorem is true for $K \subset F' \times I$.  It follows that $K$ is parallel in $F' \times I \subset F \times I$ to an essential simple closed curve in $F' \times \{0 \} \subset F \times \{ 0 \}$, as required.
\end{proof}

\section{An application: surgery on manifolds fibering over the circle}

Ying-Qing Wu has pointed out to us that the arguments above easily give this companion theorem to \ref{thm:main}:

\begin{thm} \label{thm:redmain} Suppose $F$ is a compact orientable surface, $K$ is a knot in $F \times I$, and $(F \times I)_{surg}$ is the $3$-manifold obtained by some non-trivial surgery on $K$. If $(F \times I)_{surg}$ is reducible, then either 
\begin{enumerate}

\item $K$ lies in a ball
\item $K$ is cabled and the surgery slope is that of the cabling annulus or
\item $F$ is a torus, $K$ is parallel to an essential simple closed curve in $F \times \{ 0 \}$, and the annulus that describes the parallelism determines the surgery slope.  
\end{enumerate}
\end{thm}    

\begin{proof}  As previously, let $\Fempty = F \times I - \eta(K).$  If $\Fempty$ is reducible then $K$ lies in a ball, option 1.  So henceforth we assume that $\Fempty$ is irreducible.

If $F$ has boundary, so  $F \times I$ is just a handlebody then \cite{Sch}, immediately implies that $K$ is cabled and the surgery slope is that of the cabling annulus, option 2.  

Suppose $K$ is a satellite knot, so $K$ lies in a solid torus $L \subset F \times I$ with $\bdd L$ essential in $\Fempty$.  We may as well take $L$ to be maximal with this property, so $L$ itself is not a satellite.    If some reducing sphere for $(F \times I)_{surg}$ can be isotoped inside of $L_{surg}$ then \cite{Sch} again leads to option 2.  On the other hand, if there is a reducing sphere for $(F \times I)_{surg}$ that cannot be isotoped to lie entirely inside $L_{surg}$, then the argument below could be applied to $L$ instead of $\eta(K)$. That would lead to the same contradiction with \cite{Ga3} as in Proposition \ref{prop:annulus}: the surgery slope on $\bdd L$ would be a longitude of $L$ so the winding number of $K$ in $L$ would have to be $1$.  In that case $K$ would be a core of $L$, not a satellite. Following these remarks, we are left with the case in which $F$ is closed and $K$ is not a satellite knot.

In this case, the argument of Lemma \ref{lemma:found1} (or the corresponding part of the argument in Proposition \ref{prop:torus} when $F$ is a torus) shows that there is an essential spanning annulus $A$ in $F \times I$ that is disjoint from $K$. Let $F' = F - \eta(A)$.  If $A$ is disjoint from some reducing sphere in $(F \times I)_{surg}$ then $K \subset F' \times I$, and we are done by the previous case.  On the other hand, if every reducing sphere in $(F \times I)_{surg}$ does intersect $K$, then an innermost circle argument shows that the core curve of $A$ must bound a disk in $(F \times I)_{surg}$ so in particular the end of $A$ in $F \times \{0\}$ compresses in $(F \times I)_{surg}$.  It follows then from Theorem \ref{thm:main} that $K$ is parallel to an essential simple closed curve in $F \times \{ 0 \}$ and the annulus that describes the parallelism determines the surgery slope.  This is (almost) option 3.  It remains only to show that $F$ is  a torus.  

Here is another way of viewing $(F \times I)_{surg}$ in this last case:  Let $c \subset F$ be the simple closed curve to which $K$ is parallel.  Consider the compression-body $H$ obtained by attaching a $2$-handle to $F \times I$ along $c \times \{0\}$.  Then $\bdd H$ consists of $\bdd_+H = F \times \{1\}$ and $\bdd_-H = \bdd H - \bdd_+H$, the $1$- or $2$-component surface obtained by compressing $F$ along $c$.  It is easy to see that $(F \times I)_{surg}$ can be viewed as the double of $H$ along $\bdd_-H$.  

It is elementary to check that $\bdd_- H$ is incompressible in $H$ and, unless $\bdd_- H$ is itself a sphere, $H$ is irreducible.  Hence $(F \times I)_{surg}$, the double of $H$ along $\bdd_-H$, is irreducible unless $\bdd_-H$ is a sphere.  But if $\bdd_-H$ is a sphere, then it must have been obtained by compressing a torus, so $F$ must be a torus. \end{proof}

This result, together with Theorem \ref{thm:main}, leads immediately to 

\begin{cor} \label{cor:main}   Suppose $M$ is a compact orientable $3$-manifold that fibers over the circle.  If surgery on $K \subset M$ yields a reducible manifold, then either 
\begin{enumerate}
\item the projection $K \subset M \to S^1$ has non-trivial winding number,
\item $K$ lies in a ball,
\item $K$ is cabled and the surgery slope is that of the cabling annulus, or
\item $K$ lies in a fiber and the fiber determines the surgery slope.
\end{enumerate}
\end{cor}

\begin{proof}  We assume that options 1 and 2 are not the case, so $M - \eta(K)$ is irreducible and $K$ has trivial algebraic intersection with a fiber $F$.  As in the proof of Theorem \ref{thm:redmain}, if $K$ is a satellite knot then $K$ is cabled with surgery slope that of the cabling annulus, option 3.  So henceforth we further assume that $K$ is not a satellite knot.  

Let $M_{surg}$ denote the manifold obtained from $M$ by non-trivial surgery on $K$.  We can view $M - \eta(K)$ as a taut sutured manifold and the hypothesis is that when $\bdd \eta(K)$ is filled in some non-trivial way to create $M_{surg}$, then the result is reducible.  It follows from \cite[Corollary 2.4]{Ga2} that filling $\bdd \eta(K)$ in a trivial way does not reduce the Thurston norm of $[F] \in H_2(M, \bdd M)$, so there is a surface homeomorphic to $F$ that represents $[F]$ and which is disjoint from $K$.  But in a fibered manifold such as $M$, any surface homeomorphic to $F$ and representing $[F]$ is properly isotopic in $M$ to $F$.  Put  another way, $K$ may be isotoped so that $K \subset (M - \eta(F)) \cong F \times I$.  If there is a reducing sphere for $M_{surg}$ that lies in $M_{surg} - \eta(F)$ then the result follows from Theorem \ref{thm:redmain}.  If no reducing sphere for $M_{surg}$ can be isotoped to lie in $M_{surg} - \eta(F)$ then an innermost disk argument on a reducing sphere for $M_{surg}$ shows that $F$ compresses in $M_{surg}$.  In that case, the result follows from Theorem \ref{thm:main}. \end{proof}

 \bibliography{mybibliounique}
 \bibliographystyle{plain}

\end{document}